\documentclass[11pt,reqno]{amsart}
\usepackage{hyperref}
\usepackage{url}
\usepackage{amssymb}
\usepackage[T1]{fontenc}
\usepackage{lmodern}
\usepackage[utf8]{inputenc}
\usepackage{ifthen}
\usepackage{enumerate}

\newtheorem{theorem}{Theorem}
\newtheorem{corollary}{Corollary}
\newtheorem{conj}{Conjecture}
\newtheorem{lemma}{Lemma}

\theoremstyle{definition}

\theoremstyle{remark}
\newtheorem{rem}{Remark}
\numberwithin{equation}{section}
\numberwithin{theorem}{section}
\numberwithin{lemma}{section}
\numberwithin{corollary}{section}
\numberwithin{conj}{section}

\usepackage{tikz,enumerate}
\allowdisplaybreaks[4]
\begin{document}

\title[Fractional Powers of the Partition Function]
 {Fractional Powers of the Generating Function for the Partition Function}

\author{Heng Huat Chan and Liuquan Wang}
\address{Department of Mathematics, National University of Singapore, Singapore, 119076, Singapore}
\email{matchh@nus.edu.sg}

\address{School of Mathematics and Statistics, Wuhan University, Wuhan 430072, Hubei, People's Republic of China}
\email{wanglq@whu.edu.cn; mathlqwang@163.com}

\subjclass[2010]{Primary 05A17; Secondary 11P83}

\keywords{Congruences; partitions; powers of eta function; multipartitions}

\dedicatory{}

\maketitle

\begin{abstract}
Let $p_{k}(n)$ be the coefficient of $q^n$ in the series expansion of $(q;q)_{\infty}^{k}$.
It is known that the partition function $p(n)$, which corresponds to the case when $k=-1$, satisfies
congruences such as $p(5n+4)\equiv 0\pmod{5}$. In this article, we discuss
congruences satisfied by $p_{k}(n)$ when $k$ is a rational number.
\end{abstract}

\section{Introduction}

Let $n$ be a positive integer. A partition of a positive  integer $n$ is a finite nonincreasing sequence of positive integers
 $\lambda_1, \lambda_2,\ldots, \lambda_r$ such that $$\sum_{i=1}^r \lambda_i =n.$$
We denote the number of partitions of $n$ by $p(n)$. By convention, we set $p(0)=1$.
It is well known (see \cite[Section 19.3]{Hardy-Wright}) that the generating function of $p(n)$ is
\begin{equation*}
\sum\limits_{n=0}^{\infty}p(n)q^n=\dfrac{1}{(q;q)_{\infty}},
\end{equation*}
where
\[(a;q)_{\infty}=\prod\limits_{n=0}^{\infty}(1-aq^{n}), \quad |q|<1.\]
It was observed by S.\ Ramanujan \cite{Rama1919} that $p(n)$ satisfies the congruences
\begin{align}
p(5n+4) &\equiv  0 \pmod{5}, \label{Rama-mod-5}\\
p(7n+5) &\equiv 0\pmod{7},
 \label{Rama-mod-7}\intertext{and}
p(11n+6) &\equiv 0 \pmod{11}. \label{Rama-mod-11}
\end{align}
For Ramanujan's discussion of (\ref{Rama-mod-5})--(\ref{Rama-mod-11}), see
\cite{Rama1919,Ram-congruence}.

Let $k$ be an integer and define $p_{k}(n)$ by
\begin{align}\label{pk-defn-add}
\sum_{n=0}^{\infty}p_{k}(n)q^n=(q;q)_{\infty}^{k}.
\end{align}
Observe that $p(n)=p_{-1}(n)$. When $k$ is a positive integer, $p_{-k}(n)$ enumerates the number of multipartitions with $k$ components of $n$ \cite{Andrews}. The arithmetic properties of $p_{-k}(n)$ have been extensively studied. For example, A.O.L.\ Atkin \cite{Atkin} gave a list of congruences modulo arbitrary powers of 2, 3, 5 and 7 satisfied by $p_{-k}(n)$.
B.\ Gordon \cite{Gordon} established congruences modulo arbitrary powers of 11 for $p_{-k}(n)$ for $k\in \mathbf{Z}$. From their works, we know that there are many congruences of the form
\begin{align}\label{form}
p_{-k}(\ell n+r)\equiv 0 \pmod{\ell},
\end{align}
where $\ell$ is a prime and $0 \leq r \leq \ell-1$.  I.\ Kiming and J.\ Olsson \cite{Kiming} proved that if $\ell \ge 5$ is a prime, $1\le k \le \ell-1$ and $k\notin \{\ell-3, \ell-1\}$, then a congruence of the form \eqref{form} exists only if $k$ is an odd integer and $24r+k\equiv 0$ (mod $\ell$). M.\ Boylan \cite{Boylan} has found all possible congruences of the form \eqref{form} when $k$ is a positive odd integer not exceeding 47. Recently, by using the theory of modular forms, M.\ Locus and I.\ Wagner \cite{Locus} obtained some congruences of the form \eqref{form} for positive integer $k$ with some restrictions on $\ell$ and $r$.

Around 2003, S.T. Ng \cite{Ng}, under the suggestion of the first author, considered $p_k(n)$ defined in \eqref{pk-defn-add} when $k$ is a negative rational number. He proved, using the theory of modular forms,  that for any $n\ge 0$,
\begin{equation}\label{2-3-mod-19}
p_{-2/3}(19n+9) \equiv  0 \pmod{19}.
\end{equation}
It was also mentioned in \cite{Ng,Yang} that Y.F. Yang showed in an unpublished work that for any $n\ge 0$,
\begin{equation}\label{1-2-mod-17}
p_{-1/2}(17n+11) \equiv 0 \pmod{17}.
\end{equation} 

In this article, we prove numerous congruences satisfied by $p_{k}(n)$ when $k$ is a rational number. We first introduce some notations. For any real number $x$, we denote by $\lfloor x \rfloor$ the integer part of $x$. For any integer $n$ and prime $p$, we use $\mathrm{ord}_p(n)$ to denote the integer $m$ such that $p^m|n$ and $p^{m+1}\nmid n$. For any rational number $x$, we write it in reduced form $x=u/v$ with $u,v\in \mathbf{Z}$, $\gcd(u,v)=1$ and $v\ge 1$, and we denote $\mathrm{denom} (x)=v$ to be the denominator of $x$. In the following theorem, we determine the denominator of $p_{k}(n)$.
\begin{theorem}\label{denominator}
Let $k=a/b$, where $a,b\in \mathbf{Z}$, $b\ge 1$ and $\gcd(a,b)=1$. We have
\begin{align}
\mathrm{denom} \left(p_{k}(n)\right)=b^n\prod\limits_{p|b}p^{\alpha_p(n)}
\end{align}
where
\begin{align}
\alpha_p(n)=\mathrm{ord}_{p}(n!)=\left\lfloor \frac{n}{p}\right\rfloor +\left\lfloor \frac{n}{p^2}\right\rfloor+\cdots.
\end{align}
\end{theorem}
This theorem implies that $b$ and the denominator of $p_{a/b}(n)$ share the same prime divisors.
For instance, from the series expansions
\begin{align}\label{q-1-2}
(q;q)_{\infty}^{-1/2}=&1 + \dfrac{1}{2}q + \dfrac{7}{8}q^2 + \dfrac{17}{16}q^3 + \dfrac{203}{128}q^4 + \dfrac{455}{256}q^5 + \dfrac{2723}{1024}q^6 + \dfrac{6001}{2048}q^7 + \dfrac{133107}{32768}q^8 \nonumber \\
&+\dfrac{312011}{65536}q^9 + \dfrac{1613529}{262144} q^{10}+\cdots
\end{align}
and
\begin{align}\label{qq-1-3}
(q;q)_{\infty}^{1/3}=&1 - \dfrac{1}{3}q - \dfrac{4}{9}q^2 - \dfrac{23}{81}q^3 - \dfrac{82}{243}q^4 - \dfrac{34}{729}q^5 - \dfrac{1711}{6561} q^6 + \dfrac{2254}{19683}q^7 - \dfrac{5117}{59049}q^8 \nonumber \\
 &+ \dfrac{124025}{1594323} q^9 + \dfrac{183415}{4782969}q^{10}+\cdots,
\end{align}
we observe that the denominators of $p_{-1/2}(n)$ and $p_{1/3}(n)$ are powers of 2 and 3, respectively.

From Theorem \ref{denominator}, we know that it is meaningful to study congruences modulo $m$ satisfied by $p_{a/b}(n)$ for any positive integer $m$ such that $\gcd (m,b)=1$. By using the known series expansion of $(q;q)_{\infty}^{d}$ where $d\in \{1,3,4,6,8,10, 14,26\}$, we obtain the following result:
\begin{theorem}\label{thm-general}
Suppose $a,b,d\in \mathbf{Z}$, $b\ge 1$ and $\gcd(a,b)=1$. Let $\ell$ be a prime divisor of $a+db$ and $0\leq r<\ell$.
Suppose $d,\ell$ and $r$ satisfy any of the following conditions:
\begin{enumerate}[{\rm (1)}]
\item $d=1$ and $24r+1$ is a quadratic non-residue modulo $\ell$;
\item $d=3$ and $8r+1$ is a quadratic non-residue modulo $\ell$ or $8r+1 \equiv 0$ \text{\rm{(mod $\ell$)}};
\item $d\in \{4,8,14\}$, $\ell \equiv 5$ \text{\rm{(mod 6)}} and $24r+d\equiv 0$ \text{\rm{(mod $\ell$)}};
\item $d \in \{6, 10\}$, $\ell \ge 5$ and $\ell \equiv 3$ \text{\rm{(mod 4)}} and $24r+d\equiv 0$ \text{\rm{(mod $\ell$)}};
\item $d=26$, $\ell \equiv 11$ \text{\rm{(mod $12$)}} and $24r+d\equiv 0$ \text{\rm{(mod $\ell$)}}.
\end{enumerate}
Then for $n\geq 0, $
\begin{equation}\label{general-cong}
p_{-{a}/{b}}(\ell n+r) \equiv 0 \pmod{\ell}.
\end{equation}
\end{theorem}

Let $(a,b)=(1,1)$ in Theorem \ref{thm-general}. Then by setting $(d, \ell, r)$ to be $(4, 5, 4)$, $(6, 7, 5)$ and $(10, 11, 6)$, we obtain Ramanujan's congruences \eqref{Rama-mod-5}, \eqref{Rama-mod-7} and \eqref{Rama-mod-11}, respectively.  Since the arithmetic properties of $p_{k}(n)$ when $k\in \mathbf{Z}$ have already been extensively studied, we will concentrate on the cases when $k\in \mathbf{Q}- \mathbf Z$. In this direction, Theorem \ref{thm-general} gives many explicit congruences. For example, we have
\begin{align}
p_{1/2}(11n+8) &\equiv 0 \pmod{11}, \\
p_{1/3}(41n+37) &\equiv 0 \pmod{41}, \\
p_{3/5}(59n+53) &\equiv 0 \pmod{59}, \\
p_{-1/2}(29n+26) &\equiv  0 \pmod{29},\\
p_{-1/3}(31n+28) &\equiv 0 \pmod{31},\\
p_{-3/4}(43n+39) &\equiv 0 \pmod{43},\intertext{and}
p_{-1/5}(71n+29) &\equiv 0 \pmod{71}.
\end{align}
Besides these congruences implied by Theorem \ref{thm-general}, we also discover several congruences modulo powers of primes. A sample of such congruences are as follows:
\begin{align}
p_{1/5}(7n+6) &\equiv 0 \pmod{49}, \label{intro-1}\\
p_{-1/2}(49n+r) &\equiv  0 \pmod{49}, \quad r \in \{20,34,41,48\}, \label{intor-2} \\
p_{-2/3}(49n+r) & \equiv  0 \pmod{49}, \quad r \in \{22,29,43 \}. \label{intro-3}
\end{align}

The paper is organized as follows. In Section \ref{general}, we give proofs to Theorems \ref{denominator} and \ref{thm-general}. In Section \ref{explicit}, we present many congruences satisfied by $p_{a/b}(n)$ where $1\leq |a|<b\leq 5$ modulo primes or prime powers.
Our study of functions $p_k(n)$, with negative rational numbers $k$, also leads to new proofs of
Ramanujan's congruences \eqref{Rama-mod-5} and \eqref{Rama-mod-7}.

The partition function $p(n)$ satisfies congruences associated with prime powers involving primes $5,7$ and 11. Our motivation in studying $p_k(n)$ is to find congruences modulo prime powers $\ell^s$ with primes $\ell>11$. An example of such congruences is that for any $n\ge 0$,
\begin{align}
p_{-1/2}(289n+283) \equiv 0 \pmod{289}.
\end{align}
Many other such congruences are presented as conjectures in Section \ref{explicit}.


\section{Proofs of Theorems \ref{denominator} and \ref{thm-general}}\label{general}
In this section, we prove Theorems \ref{denominator} and \ref{thm-general}.
\begin{proof}[Proof of Theorem \ref{denominator}]
Note that
\begin{equation}\label{1-b-gen}
\sum\limits_{n=0}^{\infty}p_{a/b}(n)q^n =\prod\limits_{m=1}^{\infty}(1-q^{m})^{a/b}.
\end{equation}
We deduce, using the generalized binomial theorem, that
\[(1-q^{m})^{a/b}=\sum\limits_{n=0}^{\infty}c_{a/b}(n)(-1)^nq^{mn}\]
where
\begin{align}\label{l(n)}
c_{a/b}(n)&=\dfrac{1}{n!}\dfrac{a}{b}\left(\dfrac{a}{b}-1\right)\left(\dfrac{a}{b}-2\right)\cdots \left(\dfrac{a}{b}-n+1\right)\nonumber \\
&=\dfrac{a(a-b)(a-2b)\cdots (a-(n-1)b)}{b^{n}n!}.
\end{align}
We want to show that
\begin{equation}\label{n!-cong}
a(a-b)(a-2b)\cdots (a-(n-1)b)b^{n-1} \equiv 0 \pmod{n!}.
\end{equation}

Let $\nu=\mathrm{ord}_{p}(n!)$. Since
\[\nu=\left\lfloor\dfrac{n}{p}\right\rfloor+\left\lfloor\dfrac{n}{p^2}\right\rfloor+\cdots < \dfrac{n}{2}+\dfrac{n}{4}+\cdots =n,\]
we conclude that $\nu \le n-1$. Therefore, if $p|b$, then $p^{\nu}|b^{n-1}$.

If $p\nmid b$, then for any integer $m\ge 0$, the set
$$\{a-p^tm, a-(p^tm+1)b, \cdots, a-(p^t(m+1)-1)b\}$$
forms a complete set of residues modulo $p^t$. Therefore, if $0\leq r<p^t$, then $r$ will appear $\left\lfloor\dfrac{n}{p^t}\right\rfloor$ times when the integers in the set
$$S=\{a,a-b,a-2b,\cdots, a-(n-1)b\}$$
are written in terms of their least non-negative residues modulo $p^t$. So the set $S$ contains at least $\left\lfloor \dfrac{n}{p^t} \right\rfloor$ integers divisible by $p^t$. This implies that
\[\mathrm{ord}_{p}\left(\prod\limits_{i=0}^{n-1}(a-ib) \right)\ge \sum\limits_{t\ge 1}\left\lfloor\dfrac{n}{p^{t}}\right\rfloor=\nu.\]
Therefore, we deduce that for any prime $p$, the order of $p$ dividing $n!$ cannot be greater than the order of $p$ dividing the left hand side of \eqref{n!-cong}. Hence (\ref{n!-cong}) holds.

From (\ref{l(n)}) and (\ref{n!-cong}), we find that $\mathrm{denom}(c_{a/b}(n))$ divides $b^{2n-1}$. Therefore, any prime factor of $\mathrm{denom}(c_{a/b}(n))$ divides $b$. Moreover, since $\gcd(a,b)=1$, we find that $\prod\limits_{i=0}^{n-1}(a-ib)$ is not divisible by $p$ for each prime  $p|b$.
By \eqref{l(n)}, we deduce that
\begin{align}\label{order-c-add}
\mathrm{ord}_{p}\left(\mathrm{denom}(c_{a/b}(n))\right)=n\mathrm{ord}_{p}(b)+\mathrm{ord}_{p}(n!).
\end{align}

From \eqref{1-b-gen}, we obtain
\begin{align}
p_{a/b}(n)=\sum_{\begin{smallmatrix} m_1n_1+\cdots +m_rn_r=n \\ 0<m_1<\cdots <m_r, r\ge 1\end{smallmatrix}} c_{a/b}(n_1)\cdots c_{a/b}(n_r)(-1)^{n_1+\cdots+n_r}. \label{p-convolution}
\end{align}
For each prime $p|b$, we deduce from \eqref{order-c-add} that
\begin{align}
&\mathrm{ord}_{p}\left(\mathrm{denom}\left(c_{a/b}(n_1)\cdots c_{a/b}(n_r) \right) \right) \nonumber \\
=&~(n_1+\cdots +n_r)\mathrm{ord}_{p}(b)+\sum_{i=1}^{r}\left(\left\lfloor\frac{n_i}{p} \right\rfloor+\left\lfloor \frac{n_i}{p^2}\right\rfloor +\cdots \right) \nonumber \\
\le &~(n_1+\cdots +n_r)\mathrm{ord}_{p}(b)+\left(\left\lfloor\frac{\sum_{i=1}^{r}n_i}{p}\right\rfloor+\left\lfloor \frac{\sum_{i=1}^{r}n_i}{p^2}\right\rfloor +\cdots \right) \nonumber \\
\le &~n\mathrm{ord}_{p}(b)+\left(\left\lfloor\frac{n}{p}\right\rfloor+\left\lfloor \frac{n}{p^2}\right\rfloor +\cdots \right) \label{ord-ineq-add}
\end{align}
where for the second last inequality of \eqref{ord-ineq-add}, we used the fact that
\begin{align}
\sum_{i=1}^{m}\lfloor x_i \rfloor \le \lfloor \sum_{i=1}^{m}x_i \rfloor, \quad x_1,x_2, \dots, x_m \in \mathbf{R},
\end{align}
and for the last inequality of \eqref{ord-ineq-add}, we used the fact that
\begin{align}\label{ni-ineq}
\sum_{i=1}^{r}n_i\le \sum_{i=1}^{r}m_in_i=n.
\end{align}
We observe that equality in \eqref{ord-ineq-add} holds only if the equality in \eqref{ni-ineq} holds. Since $m_1<m_2<\cdots <m_r$, we know that this happens only if $r=1$, $m_1=1$ and $n_1=n$. In this case, we do have
\begin{align}
\mathrm{ord}_{p}\left(\mathrm{denom}(c_{a/b}(n)) \right)= n\mathrm{ord}_{p}(b)+\left(\left\lfloor\frac{n}{p}\right\rfloor+\left\lfloor \frac{n}{p^2}\right\rfloor +\cdots \right).
\end{align}
Hence, in the sum on the right side of \eqref{p-convolution}, the order of $p$ of the denominator of each term is at most $n\mathrm{ord}_{p}(b)+\mathrm{ord}_{p}(n!)$ and exactly one term achieves this maximal order. Therefore, we  have
\begin{align}
\mathrm{ord}_{p}\left(\mathrm{denom}(p_{a/b}(n))\right)=n\mathrm{ord}_{p}(b)+\mathrm{ord}_{p}(n!).
\end{align}
This proves the theorem since any prime divisor of $\mathrm{denom}(p_{a/b}(n))$ also divides $b$.
\end{proof}

To prove Theorem \ref{thm-general}, we need the following lemma.
\begin{lemma}\label{binom-lem}
Let $k=a/b$, where $a,b\in \mathbf{Z}$, $b\ge 1$ and $\gcd(a,b)=1$. Let $p$ be a prime such that $p\nmid b$. We have
\begin{align}\label{fraction-cong}
(1-x)^{p^jk} \equiv (1-x^p)^{p^{j-1}k} \pmod{p^j}
\end{align}
and for any positive integer $t$,
\begin{align}\label{prod-cong}
(q^t;q^t)_\infty^{p^jk}\equiv (q^{pt};q^{pt})_{\infty}^{p^{j-1}k} \pmod{p^j}.
\end{align}
\end{lemma}
\begin{proof}
It suffices to prove \eqref{fraction-cong} since \eqref{prod-cong} follows from \eqref{fraction-cong}.

By the binomial theorem and the fact that for any $0< j < p$,
$$\binom{p}{j}\equiv 0 \pmod{p},$$
we have
\begin{align}
(1-x)^{p}=\sum_{j=0}^{p}\binom{p}{j}(-1)^jx^j \equiv 1-x^p \pmod{p}.
\end{align}
By induction on $j$, we deduce that
\[(1-x)^{p^j} \equiv (1-x^p)^{p^{j-1}} \pmod{p^j}.\]
Let
\[(1-x)^{p^j}=(1-x^p)^{p^{j-1}}+p^jF(x),\]
where $F(x)$ is a power series in $x$ with integer coefficients. From the proof of Theorem \ref{denominator}, we know that the denominator of $c_{a/b}(n)$ (in reduced form) divides $b^{2n-1}$, and hence is not divisible by $p$. Therefore,
\begin{align*}
(1-x)^{p^ja/b}&=\Big((1-x^p)^{p^{j-1}}+p^jF(x) \Big)^{a/b} \\
&=(1-x^p)^{p^{j-1}a/b}\sum_{n=0}^{\infty} c_{a/b}(n)p^{jn}\left(\dfrac{F(x)}{(1-x^p)^{p^{j-1}}}\right)^n \\
& \equiv (1-x^p)^{p^{j-1}a/b} \pmod{p^j}. \qedhere
\end{align*}
\end{proof}

We are now ready to prove Theorem \ref{thm-general}.
\begin{proof}[Proof of Theorem \ref{thm-general}] Since $\ell|(a+db)$, we may let $a+db=\ell m$ for some integer $m$.
Next, $\gcd(a,b)=1$ and $\ell|(a+db)$ implies that $\gcd(\ell,b)=1$.
Since $\gcd(\ell,b)=1$, by Lemma \ref{binom-lem}, we find that
\begin{align}\label{d-start}
\sum\limits_{n=0}^{\infty}p_{-a/b}(n)q^n=\dfrac{(q;q)_{\infty}^{d}}{(q;q)_{\infty}^{(a+db)/b}}\equiv \dfrac{(q;q)_{\infty}^{d}}{(q^\ell;q^\ell)_{\infty}^{{m/b}}} \pmod{\ell}.
\end{align}
We now present our proof with respect to the values of $d$.

{\bf Case} $d=1$: By Euler's pentagonal number theorem \cite[Corollay 1.3.5]{Berndt}, we find that
\begin{align}\label{eta-1}
(q;q)_{\infty}=\sum_{i=-\infty}^{\infty}(-1)^iq^{i(3i+1)/2}.
\end{align}
Note that
\begin{align}\label{d=1}
N=\dfrac{i(3i+1)}{2}
\end{align}
is equivalent to
\begin{align}
24N+1=(6i+1)^2.
\end{align}
 Therefore, if $24N+1$ is a quadratic non-residue modulo $\ell$, then there are no integers $i$ satisfying \eqref{d=1}. The congruence \eqref{general-cong} follows by comparing the coefficients of $q^{\ell n+r}$  on both sides of \eqref{d-start}.

{\bf Case} $d=3$:   By Jacobi's identity \cite[Theorem 1.3.9]{Berndt}, we find that
\begin{align}\label{eta-3}
(q;q)_{\infty}^3=\sum\limits_{j=0}^{\infty}(-1)^j(2j+1)q^{j(j+1)/2}.
\end{align}
Note that
\begin{align}\label{3-mid}
N=\dfrac{j(j+1)}{2}
\end{align}
is equivalent to
\begin{align}\label{3-end}
8N+1=(2j+1)^2.
\end{align}
If $8N+1$ is a quadratic non-residue modulo $\ell$, then there are no integers $j$ satisfying \eqref{3-mid}. Hence from \eqref{d-start} we conclude that $p_{-{a}/{b}}(\ell n+r) \equiv 0$ (mod $\ell$).

If $8N+1\equiv 0$ (mod $\ell$), then \eqref{3-end} implies that $2j+1\equiv 0$ (mod $\ell$). Again by \eqref{eta-3} and \eqref{d-start}, we deduce \eqref{general-cong}.

{\bf Case} $d=4$: From \eqref{eta-1} and \eqref{eta-3}, we find that
\begin{align}\label{eta-4}
(q;q)_{\infty}^4=\sum_{i=-\infty}^{\infty}\sum_{j=0}^{\infty}(-1)^{i+j}(2j+1)q^{i(3i+1)/2+j(j+1)/2}.
\end{align}
Now, observe that
\[N=\dfrac{i(3i+1)}{2}+\dfrac{j(j+1)}{2}
\]
if and only if \[24N+4=(6i+1)^2+3(2j+1)^2.\]
If $\ell \equiv 5$ (mod 6), then $\left(\dfrac{-3}{\ell}\right)=-1$. This implies that
\[24N+4\equiv 0 \pmod{\ell}\]
if and only if
\[6i+1\equiv 0 \pmod{\ell}\,\,\text{ and}\,\, 2j+1\equiv 0 \pmod{\ell}.\]
Using \eqref{eta-4} and comparing the coefficients of $q^{\ell n+r}$  on both sides of \eqref{d-start}, we obtain \eqref{general-cong}.

{\bf Case} $d=6$: We deduce from  \eqref{eta-3} that
\begin{align}\label{eta-6}
(q;q)_{\infty}^6=\sum_{i=0}^{\infty}\sum_{j=0}^{\infty}(-1)^{i+j}(2i+1)(2j+1)q^{i(i+1)/2+j(j+1)/2}.
\end{align}
Observe that
\[N=\dfrac{i(i+1)}{2}+\dfrac{j(j+1)}{2}\]
is equivalent to
 \[8N+2=(2i+1)^2+(2j+1)^2.\]
If $\ell \equiv 3$ (mod 4), then $\left(\dfrac{-1}{\ell}\right)=-1$. This implies that \[8N+2\equiv 0 \pmod{\ell}\]
if and only if \[2i+1\equiv 0 \pmod{\ell}\,\,\text{and}\,\, 2j+1\equiv 0 \pmod{\ell}.\] Congruence \eqref{general-cong} follows by comparing the coefficients of $q^{\ell n+r}$ on both sides of \eqref{d-start}.

{\bf Case} $d=8$: We need the following identity (see \cite{Lin-2015Rama})
\begin{align}\label{eta-8}
(q;q)_{\infty}^{8}=&\dfrac{4}{3}\left(\sum\limits_{m=-\infty}^{\infty}(3m+1)^3q^{3m^2+2m}  \right)\left(\sum\limits_{n=-\infty}^{\infty}q^{n^2} \right)\nonumber \\
&-\dfrac{1}{3}\left(\sum\limits_{m=-\infty}^{\infty}(6m+1)^3q^{3m^2+m}  \right)\left(\sum\limits_{n=0}^{\infty}q^{n^2+n} \right).
\end{align}
Note that
\[N=3m^2+2m+n^2\]
is equivalent to \[3N+1=(3m+1)^2+3n^2.\]
Suppose $3m+1$ and $n$ are non-zero modulo $\ell$.
Then $3N+1\equiv 0 \pmod{\ell}$ implies
$$u^2\equiv -3 \pmod{\ell}$$ for some integer $u$.
But since $\ell \equiv 5$ (mod 6), we have $\left(\dfrac{-3}{\ell}\right)=-1$, and hence such an integer $u$ cannot exist. Therefore,
if $3m+1$ and $n$ are non-zero modulo $\ell$, then $3N+1$ is non-zero modulo $\ell$.
In other words, \[3N+1\equiv 0\pmod{\ell}\]
if and only if \[3m+1\equiv 0\pmod{\ell}\,\,\text{and}\,\, n\equiv 0\pmod{\ell}.\]
Similarly, note that
\[N=3m^2+m+n^2+n\] is equivalent to \[4(3N+1)=(6m+1)^2+3(2n+1)^2.\]
This identity implies, as in the previous case, that
\[3N+1\equiv 0 \pmod{\ell}\]
if and only if
\[6m+1\equiv 0\pmod{\ell}\text{\,\, and\,\,} 2n+1 \equiv 0\pmod{\ell}.\]
Therefore, from \eqref{d-start} and \eqref{eta-8} we see that $p_{-a/b}(\ell n+r)\equiv 0$ (mod $\ell$).

{\bf Case} $d=10$: From \cite[Corollary 4.2]{Chu}, we find that
\begin{align}\label{eta-10}
(q;q)_{\infty}^{10}=&\dfrac{4}{3}\left(\sum_{m=-\infty}^{\infty}(3m+1)^3q^{3m^2+2m}  \right)\times \left(\sum_{n=-\infty}^{\infty}(6n+1)q^{3n^2+n} \right)\nonumber \\
&-\left(\sum_{m=-\infty}^{\infty}(3m+1)q^{3m^2+2m} \right) \times \left(\sum_{n=-\infty}^{\infty}(6n+1)^3q^{3n^2+n} \right).
\end{align}
Observe that
\[N=3m^2+2m+3n^2+n\] is equivalent to \[12N+5=(6m+2)^2+(6n+1)^2.\]
If $\ell \equiv 3$ (mod 4), then  $\left(\dfrac{-1}{\ell} \right)=-1$. We know that
\[12N+5\equiv 0 \pmod{\ell}\]
if and only if
\[3m+1\equiv 0\pmod{\ell}\,\,\text{ and}\,\, 6n+1\equiv 0\pmod{\ell}.\] From \eqref{eta-10}, congruence \eqref{general-cong} follows by comparing the coefficients of $q^{\ell n+r}$ on both sides of \eqref{d-start}.

{\bf Case} $d=14$: Recall from  \cite[Theorem 5.3]{Chan-JLMS} that
\begin{align}\label{eta-14}
(q;q)_{\infty}^{14}=-\dfrac{1}{15}\sum_{m=-\infty}^{\infty}&\sum_{n=-\infty}^{\infty}(-1)^m(3m+1)(4n+1)(6m+4n+3)(6m-4n+1)\nonumber \\&(6m+12n+5)(6m-12n-1)q^{\left(4(3m+1)^2+3(4n+1)^2-7 \right)/12}.
\end{align}
We observe that
\[
N=(4(3m+1)^2+3(4n+1)^2-7)/12\]
is equivalent to \[12N+7=4(3m+1)^2+3(4n+1)^2.\]
If $\ell \equiv 5$ (mod 6), then $\left(\dfrac{-3}{\ell}\right)=-1$. We deduce that
\[12N+7\equiv 0 \pmod{\ell}\]
 if and only if \[3m+1\equiv 0 \pmod{\ell}\,\,\text{and}\,\, 4n+1\equiv 0\pmod{\ell}.\] The congruence $p_{-{a}/{b}}(\ell n+r) \equiv 0$ (mod $\ell$) now follows from the comparison of the coefficients of $q^{\ell n+r}$ on both sides of \eqref{d-start}.

{\bf Case} $d=26$: Let
\begin{align}\label{f-defn}
f(m,n)=\sum_{j=0}^{12}\binom{12}{2j}(-1)^jm^jn^{6-j}.
\end{align}
From \cite[Theorem 3]{Chan-Adv}, we find
\begin{align}\label{eta-26}
(q;q)_{\infty}^{26}=&\dfrac{q^{-13/12}}{16308864}\Bigg(\sum_{i=-\infty}^{\infty}\sum_{j=-\infty}^{\infty}(-1)^{i+j}f\left(\dfrac{(6i+1)^2}{2},\dfrac{(6j+1)^2}{2} \right)q^{\left((6i+1)^2+(6j+1)^2 \right)/24} \nonumber \\
&+\sum_{i=-\infty}^{\infty}\sum_{j=-\infty}^{\infty}(-1)^{i+j}f\left(12i^2,(6j+1)^2 \right)q^{i^2+(6j+1)^2/12}\Bigg).
\end{align}
Observe that
\begin{align*}
N=\dfrac{1}{24}\left((6i+1)^2+(6j+1)^2-26\right)
\end{align*}
is equivalent to
\begin{align*}
24N+26=(6i+1)^2+(6j+1)^2.
\end{align*}
If $\ell \equiv 11$ (mod 12), then $\left(\dfrac{-1}{\ell}\right)=-1$. Hence
\[24N+26\equiv 0\!\!\!\!\! \pmod{\ell}\]
if and only if
\[6i+1\equiv 0 \!\!\!\!\!\pmod{\ell}\,\text{ and }\, 6j+1\equiv 0\!\!\!\!\!\pmod{\ell},\]
in which case
\begin{align*}
f\left(\dfrac{(6i+1)^2}{2},\dfrac{(6j+1)^2}{2} \right)\equiv 0  \pmod{\ell^{12}}.
\end{align*}

Similarly, we observe that
\begin{align*}
N=\dfrac{1}{12}\left(12i^2+(6j+1)^2-13 \right)
\end{align*}
is equivalent to
\begin{align*}
12N+13 =12i^2+(6j+1)^2.
\end{align*}
If $\ell \equiv 11$ (mod 12), then $\left(\dfrac{-12}{\ell}\right)=-1$. Hence
\[12N+13\equiv 0\!\!\!\!\!\pmod{\ell}\,\]
if and only if
\[i\equiv 0\!\!\!\!\!\pmod{\ell}\,\text{\, and\,\,}
 6j+1\equiv 0\!\!\!\!\!\pmod{\ell},\]
in which case
\begin{align*}
f\left(12i^2,(6j+1)^2 \right)\equiv 0  \pmod{\ell^{12}}.
\end{align*}

Note that $16308864=2^7\cdot 3^4 \cdot 11^2 \cdot 13$. Using \eqref{eta-26} and comparing the coefficients of $q^{\ell n+r}$ on both sides of \eqref{d-start}, we obtain \eqref{general-cong}.
\end{proof}
\begin{rem}
For $d\in \{4,6,8,10,14,26\}$, there are some other double series expressions for $(q;q)_{\infty}^{d}$. For example, formulas similar to \eqref{eta-10} for $(q;q)_{\infty}^{10}$ can be found in works of B.C.\ Berndt et al. \cite{BCLY}, S.H.\ Chan \cite{SHChan}, M.D.\ Hirschhorn \cite{Hirschhorn-AJC,Hirschhorn-Book} and L.\ Winquist \cite{Winquist}. These formulas for $(q;q)_{\infty}^{10}$ were also discussed in \cite{Chu-Yan}.
\end{rem}


\section{Explicit Congruences for $p_{k}(n)$ where $k \in \mathbf{Q}\backslash \mathbf{Z}$}\label{explicit}
In this section, we give explicit congruences for $p_{-{a}/{b}}(n)$ with $1\leq |a|<b\leq 5$. Most of these congruences are special cases of Theorem \ref{thm-general} but there are some congruences which require more technical arguments.

First, we present some explicit congruences satisfied by $p_{-{a}/{b}}(n)$ where $1\leq -a < b \leq 5$.
\begin{theorem}\label{explicit-cong-1}
For any integer $n\ge 0$,
\begin{align}
p_{{1}/{2}}(5n+r)& \equiv 0 \pmod{5}, \quad r\in \{2,3,4\}, \label{pp-1-2-mod5} \\
p_{1/2}(11n+8) &\equiv 0 \pmod{11}, \label{pp-1-2-mod11} \\
p_{1/2}(19n+17) &\equiv 0 \pmod{19}, \label{pp-1-2-mod19} \\
p_{1/3}(11n+9) &\equiv 0 \pmod{11}, \label{pp-1-3-mod11}\\
p_{1/3}(17n+4) &\equiv 0 \pmod{17}, \label{pp-1-3-mod17} \\
p_{1/3}(23n+15) &\equiv 0 \pmod{23}, \label{pp-1-3-mod23} \\
p_{1/3}(41n+37) &\equiv 0 \pmod{41}, \label{pp-1-3-mod41} \\
p_{2/3}(5n+4) &\equiv 0 \pmod{5}, \label{pp-2-3-mod5} \\
p_{2/3}(7n+r) &\equiv 0 \pmod{7}, \quad r \in \{2, 4, 5, 6\},\label{pp-2-3-mod7} \\
p_{2/3}(11n+7) &\equiv 0 \pmod{11}, \label{pp-2-3-mod11}\\
p_{1/4}(5n+4) &\equiv 0 \pmod{5}, \label{pp-1-4-mod5} \\
p_{1/4}(11n+r) &\equiv 0 \pmod{11}, \quad r \in \{2,4,5,7,9,10\},\label{pp-1-4-mod11} \\
p_{1/4}(23n+17) &\equiv 0 \pmod{23}, \label{pp-1-4-mod23} \\
p_{3/4}(7n+5) &\equiv 0 \pmod{7}, \label{pp-3-4-mod7} \\
p_{3/4}(29n+19) &\equiv 0 \pmod{29}, \label{pp-3-4-mod29}\\
p_{3/4}(53n+48) &\equiv 0 \pmod{53}, \label{pp-3-4-mod53} \\
p_{1/5}(7n+r) &\equiv 0 \pmod{7}, \quad r\in \{2,4,5,6\}, \label{pp-1-5-mod7} \\
p_{1/5}(7n+6)&\equiv 0 \pmod{49}, \label{pp-1-5-mod49} \\
p_{1/5}(23n+9) &\equiv 0 \pmod{23}, \label{pp-1-5-mod23} \\
p_{2/5} (7n+5) &\equiv 0 \pmod{7}, \label{pp-2-5-mod7} \\
p_{2/5}(13n+r) &\equiv 0 \pmod{13}, \quad r\in \{4,5,7,8,9,11,12\}, \label{pp-2-5-mod13} \\
p_{2/5}(17n+15)&\equiv 0\pmod{17}, \label{pp-2-5-mod17} \\
p_{3/5}(17n+14) &\equiv 0\pmod{17}, \label{pp-3-5-mod17} \\
p_{3/5}(47n+27) &\equiv 0 \pmod{47}, \label{pp-3-5-mod47} \\
p_{3/5}(59n+53) &\equiv 0 \pmod{59}, \label{pp-3-5-mod59} \\
p_{4/5}(11n+r) &\equiv 0 \pmod{11}, \quad r\in\{2,4,5,7,8,9\}, \label{pp-4-5-mod11}\\
p_{4/5}(23n+13) &\equiv 0 \pmod{23}. \label{pp-4-5-mod23}
\end{align}
\end{theorem}
\begin{proof}
Except for congruence \eqref{pp-1-5-mod49}, other congruences follow from Theorem \ref{thm-general} with suitable parameters given in Table \ref{parameter-1}.
\begin{table}[h]
\caption{}\label{parameter-1}
\begin{tabular}{|c|c|c|c|c|c|c|c|c|c|}
  \hline
  Eq. & \eqref{pp-1-2-mod5}  & \eqref{pp-1-2-mod11} & \eqref{pp-1-2-mod19} & \eqref{pp-1-3-mod11} & \eqref{pp-1-3-mod17} & \eqref{pp-1-3-mod23} & \eqref{pp-1-3-mod41} & \eqref{pp-2-3-mod5} &\eqref{pp-2-3-mod7}  \\
  \hline
  $a$& -1 & -1 & -1 & -1 & -1       & -1 &  -1 & -2 &-2    \\
  \hline
  $b$ & 2 & 2 & 2 & 3 & 3      & 3 &  3 & 3 &3 \\
  \hline
  $d$ & 3 & 6 & 10 & 4 & 6     & 8 & 14 & 4 &3 \\
  \hline
 $\ell$ & 5 & 11 & 19 & 11 & 17 & 23 & 41 & 5 &7 \\
  \hline
  Eq. &\eqref{pp-2-3-mod11} &\eqref{pp-1-4-mod5} &\eqref{pp-1-4-mod11} &\eqref{pp-1-4-mod23} &\eqref{pp-3-4-mod7}
   &\eqref{pp-3-4-mod29} &\eqref{pp-3-4-mod53} &\eqref{pp-1-5-mod7}  &\eqref{pp-1-5-mod23} \\
   \hline

  \hline
  $a$  &  -2 & -1 &-1 &-1 & -3   &-3  &-3   &-1 &-1      \\
  \hline
  $b$ & 3 & 4  &4  &4  &4    &4  &4   &5 &5      \\
  \hline
  $d$ & 8 &4      &3 &6  &6   &8 &14   &3 &14         \\
  \hline
  $\ell$ & 11 &5 &11 &23 &7 &29 &53  &7 &23         \\
  \hline
  Eq.  &\eqref{pp-2-5-mod7} &\eqref{pp-2-5-mod13} &\eqref{pp-2-5-mod17} & \eqref{pp-3-5-mod17}  &\eqref{pp-3-5-mod47}  &\eqref{pp-3-5-mod59}  &\eqref{pp-4-5-mod11}  &\eqref{pp-4-5-mod23} & \\
  \hline
  $a$ &-2  &-2  &-2  &-3  &-3  &-3  &-4   &-4 & \\
  \hline
  $b$ &5  &5  &5  &5  &5  &5  &5    &5 & \\
  \hline
  $d$ & 6 &3  &14     &4  &10  &26 &3     &10 &\\
  \hline
  $\ell$ & 7   &13  &17  &17 &47 &59  &11 &23 & \\
  \hline
\end{tabular}
\end{table}

We now prove \eqref{pp-1-5-mod49}. By Lemma \ref{binom-lem}, we find that
\begin{align}\label{pp-mod49-start}
\sum_{n=0}^{\infty}p_{1/5}(n)q^n=\dfrac{(q;q)_{\infty}^{10}}{(q;q)_{\infty}^{49/5}}\equiv \dfrac{(q;q)_{\infty}^{10}}{(q^7;q^7)_{\infty}^{7/5}} \pmod{49}.
\end{align}
Observe that
\[N=3m^2+2m+3n^2+n\]
is equivalent to
\[ 12N+5=(6m+2)^2+(6n+1)^2.\]
Since $\left(\dfrac{-1}{7} \right)=-1$, we know that
\[12N+5\equiv 0\!\!\!\!\!\pmod{7}\,\]
if and only if
\[3m+1\equiv 0\!\!\!\!\!\pmod{7}\,\text{\, and\,}\,
 6n+1\equiv 0\!\!\!\!\!\pmod{7}.\]
 Using \eqref{eta-10} and comparing the coefficients of $q^{7n+6}$ on both sides of \eqref{pp-mod49-start}, we obtain \eqref{pp-1-5-mod49}.
\end{proof}

Numerical evidences suggest that the following congruences hold.
\begin{conj}\label{conj-1}
For any integer $n\ge 0$,
\begin{align}
p_{1/2}(125n+r) &\equiv 0 \pmod{25}, \quad r \in \{38, 63, 88, 113 \}, \label{pp-1-2-mod25} \\
p_{2/3}(25n+r) &\equiv 0 \pmod{25}, \quad r\in\{19, 24\}, \label{pp-2-3-mod25} \\
p_{2/3}(121n+84) &\equiv 0 \pmod{121}, \label{pp-2-3-mod121} \\
p_{1/4}(25n+r) &\equiv 0 \pmod{25}, \quad r\in\{14, 24\}, \label{pp-1-4-mod25} \\
p_{1/4}(25n+19)&\equiv 0 \pmod{125}, \label{pp-1-4-mod125} \\
p_{1/4}(121n+92) &\equiv 0 \pmod{121}, \label{pp-1-4-mod121} \\
p_{1/5}(49n+r) &\equiv 0 \pmod{343}, \quad r\in\{27, 34, 48\}, \label{pp-1-5-mod343} \\
p_{2/5}(49n+40) &\equiv 0 \pmod{49}. \label{pp-2-5-mod343}
\end{align}
\end{conj}

Next, we present some explicit congruences satisfied by $p_{-{a}/{b}}(n)$ where $1\leq a <b \leq 5$.
\begin{theorem}\label{explicit-cong}
For any integer $n\ge 0$,
\begin{align}
  p_{-1/2}(7n+r) &\equiv 0 \pmod{7}, \quad r \in \{2,4,5,6\}, \label{p-1-2-mod7} \\
 p_{-1/2}(49n+r)  &\equiv 0 \pmod{49},  \quad r \in \{20,34,41,48\},  \label{p-1-2-mod49} \\
  p_{-1/2}(17n+11) &\equiv 0 \pmod{17}, \label{p-1-2-mod17} \\
 p_{-1/2}(29n+26) &\equiv 0 \pmod{29}, \label{p-1-2-mod29} \\
  p_{-1/3}(5n+r) &\equiv 0 \pmod{5},  \quad r\in \{2, 3, 4\}, \label{p-1-3-mod5} \\
  p_{-1/3}(5n+3) &\equiv 0 \pmod{25}, \label{p-1-3-mod5square}\\
 p_{-1/3}(19n+14) &\equiv 0 \pmod{19},  \label{p-1-3-mod19}\\
  p_{-1/3}(31n+28) &\equiv 0 \pmod{31}, \label{p-1-3-mod31} \\
 p_{-{2}/{3}}(5n+r) &\equiv  0 \pmod{5}, \quad r\in \{3,4\}, \label{p-2-3-mod5} \\
 p_{-{2}/{3}}(11n+r) &\equiv 0 \pmod{11}, \quad r\in \{2,4,5,7,8,9\}, \label{p-2-3-mod11} \\
  p_{-1/4}(5n+r) &\equiv  0 \pmod{5},  \quad r\in \{3, 4\}, \label{p-1-4-mod5} \\
  p_{-1/4}(13n+r) &\equiv 0 \pmod{13}, \quad r \in \{4,5,7,8,9,11,12\},  \label{p-1-4-mod13} \\
 p_{-{3}/{4}}(5n+r) &\equiv  0 \pmod{5},  \quad r\in \{2, 3, 4\},  \label{p-3-4-mod5}\\
  p_{-{3}/{4}}(43n+39) &\equiv 0 \pmod{43}, \label{p-3-4-mod43} \\
 p_{-{3}/{4}}(59n+24) &\equiv 0  \pmod{59}, \label{p-3-4-mod59} \\
 p_{-{3}/{4}}(107n+97) &\equiv 0 \pmod{107}, \label{p-3-4-mod107} \\
 p_{-{1}/{5}}(31n+23) &\equiv 0 \pmod{31},  \label{p-1-5-mod31} \\
 p_{-{1}/{5}}(71n+29) &\equiv 0 \pmod{71},  \label{p-1-5-mod71} \\
 p_{-{1}/{5}}(131n+119) &\equiv 0 \pmod{131}, \label{p-1-5-mod131} \\
 p_{-2/5}(7n+r) &\equiv 0 \pmod{7},  \quad r\in \{3,4,6\},  \label{p-2-5-mod7} \\
 p_{-2/5}(11n+9) &\equiv 0 \pmod{11},  \label{p-2-5-mod11} \\
 p_{-2/5}(17n+r) &\equiv 0 \pmod{17}, \quad r\in \{2,5,7,8,9,12,13,14,16\}, \label{p-2-5-mod17} \\
p_{-3/5}(11n+8) &\equiv 0 \pmod{11}, \label{p-3-5-mod11} \\
p_{-4/5}(11n+7) &\equiv 0 \pmod{11}, \label{p-4-5-mod11}  \\
p_{-4/5}(19n+r) &\equiv 0 \pmod{19}, \quad r\in \{4,5, 7, 8, 11, 12, 13, 14, 16, 18\}.  \label{p-4-5-mod19}
\end{align}
\end{theorem}
\begin{proof}
Except for the congruences \eqref{p-1-2-mod49} and \eqref{p-1-3-mod5square}, other congruences follow directly from Theorem \ref{thm-general} with suitable parameters given in Table \ref{parameter-2}.
\begin{table}[h]
\caption{}\label{parameter-2}
\begin{tabular}{|c|c|c|c|c|c|c|c|c|}
  \hline
  Eq. & \eqref{p-1-2-mod7}  & \eqref{p-1-2-mod17} & \eqref{p-1-2-mod29} & \eqref{p-1-3-mod5} & \eqref{p-1-3-mod19} & \eqref{p-1-3-mod31} & \eqref{p-2-3-mod5} & \eqref{p-2-3-mod11} \\
  \hline
  $a$& 1 & 1 & 1 & 1 & 1       & 1 &  2 & 2    \\
  \hline
  $b$ & 2 & 2 & 2 & 3 & 3      & 3 &  3 & 3  \\
  \hline
  $d$ & 3 & 8 & 14 & 3 & 6     & 10 & 1 & 3  \\
  \hline
 $\ell$ & 7 & 17 & 29 & 5 & 19 & 31 & 5 & 11 \\
  \hline
  Eq. &\eqref{p-1-4-mod5} &\eqref{p-1-4-mod13} &\eqref{p-3-4-mod5} &\eqref{p-3-4-mod43}
   &\eqref{p-3-4-mod59} &\eqref{p-3-4-mod107} &\eqref{p-1-5-mod31}  &\eqref{p-1-5-mod71} \\
   \hline

  \hline
  $a$  &  1 &1 &3 &3 & 3   &3  &1   &1       \\
  \hline
  $b$ & 4 &4 &4  &4  &4    &4  &5   &5      \\
  \hline
  $d$ & 1 &3 &3 &10  &14   &26 &6   &14         \\
  \hline
  $\ell$ & 5 &13 &5 &43 &59 &107 &31  &71         \\
  \hline
  Eq. &\eqref{p-1-5-mod131} &\eqref{p-2-5-mod7}  & \eqref{p-2-5-mod11}  &\eqref{p-2-5-mod17}  &\eqref{p-3-5-mod11}  &\eqref{p-4-5-mod11}  &\eqref{p-4-5-mod19} & \\
  \hline
  $a$ &1  &2  &2  &2  &3  &4  &4   & \\
  \hline
  $b$ &5  &5  &5  &5  &5  &5  &5    & \\
  \hline
  $d$ &26  &1  &4  &3  &6  &8  &3     &\\
  \hline
  $\ell$ &131   &7  &11 &17 &11 &11 &19  &  \\
  \hline
\end{tabular}
\end{table}

Now we prove \eqref{p-1-2-mod49}.
By Lemma \ref{binom-lem},
\begin{align}\label{1-2-mod7power-start}
\sum\limits_{n=0}^{\infty}p_{-1/2}(n)q^{n+1}&=\dfrac{q(q;q)_{\infty}^{24}}{(q;q)_{\infty}^{49/2}} \nonumber \\
&\equiv  \dfrac{1}{(q^7;q^7)_{\infty}^{7/2}}\sum\limits_{n\ge 0}\tau(n)q^n \pmod{49},
\end{align}
where $\tau(n)$ is Ramanujan's tau function.

For any prime $p$, it is known that \cite[Chapter VII]{Serre}
\[\tau(pn)=\tau(p)\tau(n)-p^{11}\tau(n/p).\]
Hence,
\begin{equation}\label{tau-7n-n}\tau(7n)=-16744\tau(n)-7^{11}\tau(n/7) \equiv 14\tau(n) \pmod{49}.\end{equation}

Extracting the terms of the form $q^{7n}$ on both sides of \eqref{1-2-mod7power-start}, replacing $q^7$ by $q$ and using
\eqref{tau-7n-n}, we deduce that
\begin{equation*}
\sum_{n=0}^{\infty}p_{-1/2}(7n+6)q^{n+1}\equiv\dfrac{1}{(q;q)_{\infty}^{7/2}}\sum_{n=0}^{\infty}\tau(7n)q^n
\equiv\dfrac{14}{(q;q)_{\infty}^{7/2}}\sum_{n=0}^{\infty}\tau(n)q^n \pmod{49},
\end{equation*}
which implies, by Lemma \ref{binom-lem}, that
\begin{equation*}
\sum_{n=0}^{\infty} \frac{1}{7}p_{-1/2}(7n+6)q^{n+1}
\equiv\dfrac{2}{(q;q)_{\infty}^{7/2}}\sum_{n=0}^{\infty}\tau(n)q^n
\equiv\dfrac{2}{(q^7;q^7)_{\infty}^{1/2}}\sum_{n=0}^{\infty}\tau(n)q^n \pmod{7}.
\end{equation*}
Hence,
\begin{equation}\label{1-2-mod7power-end}
(q^7;q^7)_{\infty}^{1/2}\sum_{n=0}^{\infty} p_{-1/2}(7n+6)q^{n+1}\equiv
14\sum_{n=0}^{\infty}\tau(n)q^n \pmod{49}.
\end{equation}


We recall from \cite[p.\ 97, Eq.\ (56)]{Serre} that
\begin{align}\label{tau-mod7}
\tau(n) \equiv n \sigma_{3}(n) \pmod{7},
\end{align}
where $\sigma_{3}(n)=\sum_{d|n}d^3$. We claim that if the residue of $n$ modulo 7 is 3, 5 or 6, then $\sigma_{3}(n)\equiv 0$ (mod 7). Indeed, in these cases, $n$ cannot be a square number and $n^3\equiv -1$ (mod 7). Therefore, when $n\equiv 3, 5$ or $6\pmod{7}$, we find that
\[\sigma_{3}(n)=\sum_{\substack{d|n\\  d<\sqrt{n}}}\left( d^3+\left(\dfrac{n}{d}\right)^3 \right)
=\sum_{\substack{d|n\\  d<\sqrt{n}}}\dfrac{d^6-1}{d^3}\equiv 0 \pmod{7},
\]
where the last congruence follows from Fermat's little theorem.

By \eqref{tau-mod7} we deduce that
\[\tau(7n+s) \equiv 0 \pmod{7}, \quad s\in \{0, 3,5,6\}.\]
Using this result and \eqref{1-2-mod7power-end}, we deduce that
\[p_{-1/2}(7(7n+s-1)+6) \equiv 0 \pmod{49}, \quad s\in \{0, 3, 5, 6\}.\]
The congruences in \eqref{p-1-2-mod49} are proved.

Next, we prove \eqref{p-1-3-mod5square}.

By Lemma \ref{binom-lem}, we find that
\begin{align}\label{p-1-3-start}
\sum_{n=0}^{\infty}p_{-1/3}(n)q^n=\dfrac{(q;q)_{\infty}^8}{(q;q)_{\infty}^{25/3}}\equiv \dfrac{(q;q)_{\infty}^8}{(q^5;q^5)_{\infty}^{5/3}} \pmod{25}.
\end{align}
Now we use  the the expansion \eqref{eta-8} of $(q;q)_{\infty}^8$. Observe that
\[N=3m^2+2m+n^2 \]
is equivalent to
\[3N+1=(3m+1)^2+3n^2.\]
Since $\left(\dfrac{-3}{5}\right)=-1$, we find that
\[3N+1\equiv 0 \pmod{5} \,\, \text{(or equivalently, $N\equiv 3\!\!\!\!\!\pmod{5}$)}\]
{if and only if}
\[3m+1\equiv 0\pmod{5}\,\text{ and}\,  n\equiv 0\pmod{5}.\]
Similarly, observe that
\[N=3m^2+m+n^2+n\]
is equivalent to
\[4(3N+1)=(6m+1)^2+3(2n+1)^2.\]
We know that
\[3N+1\equiv 0\!\!\!\!\!\pmod{5}\text{\, (or equivalently,  $N\equiv 3\!\!\!\!\!\pmod{5}$)}\]
{if and only if}
 \[6m+1\equiv 0\!\!\!\!\!\pmod{5}\,\text{ and }\,
 2n+1\equiv 0\!\!\!\!\!\pmod{5}.\] Therefore, \eqref{eta-8} and \eqref{p-1-3-start} imply
\[p_{-1/3}(5n+3) \equiv 0 \pmod{25}. \qedhere \]
\end{proof}

As an interesting application of the congruences in this section, using \eqref{p-1-3-mod5} and \eqref{p-2-3-mod5}, we can give a new proof of \eqref{Rama-mod-5}.
\begin{corollary}
For any integer $n\ge 0$,
\[p(5n+4) \equiv 0 \pmod{5}.\]
\end{corollary}
\begin{proof}
Since
\[\sum_{n=0}^{\infty}p(n)q^n=\Big(\sum_{n=0}^{\infty}p_{-1/3}(n)q^n \Big)\Big(\sum_{n=0}^{\infty}p_{-2/3}(n)q^n \Big),\]
\begin{align}\label{p-3-product}
p(n)=\sum_{k=0}^{n}p_{-1/3}(k)p_{-2/3}(n-k).
\end{align}
Note that for any integers $k$ and $n$, either the least non-negative residue of $k$ modulo 5 belongs to $\{2,3,4\}$ or the least non-negative
residue of $5n+4-k$ modulo 5 belongs to $\{3,4\}$. Hence by  \eqref{p-1-3-mod5} and  \eqref{p-2-3-mod5}, we always have
\[p_{-1/3}(k)p_{-2/3}(5n+4-k)\equiv 0 \pmod{5}.\]
This proves the corollary.
\end{proof}

Similarly, by using \eqref{p-1-2-mod7} we give a new proof of \eqref{Rama-mod-7}.
\begin{corollary}
For any integer $n\ge 0$,
\[p(7n+5) \equiv 0 \pmod{7}.\]
\end{corollary}

\begin{proof}
Since
\[\sum\limits_{n=0}^{\infty}p(n)q^n=\Big(\sum\limits_{n=0}^{\infty}p_{-1/2}(n)q^n \Big)^2,\]
\begin{align}\label{p-product}
p(n)=\sum\limits_{k=0}^{n}p_{-{1/2}}(k)p_{-{1/2}}(n-k).
\end{align}
Note that for any integers $k$ and $n$, at least one of $k$ or $7n+5-k$ must be congruent to 2, 4, 5 or 6. By  \eqref{p-1-2-mod7} and \eqref{p-product}, we conclude that $p(7n+5)$ is always divisible by 7.
\end{proof}

Numerical evidences suggest that the following conjecture holds.
\begin{conj}\label{conj-2}
For any integer $n\ge 0$, we have
\begin{align}
p_{-1/2}(343n+293) &\equiv 0 \pmod{343}, \label{p-1-2-mod343}\\
p_{-1/2}(2401n+r) &\equiv 0 \pmod{2401}, \quad r \in \{979, 1665, 2008, 2351\}, \label{p-1-2-mod2401}\\
p_{-1/2}(289n+283) &\equiv 0\pmod{289}, \label{p-1-2-mod289} \\
p_{-1/3}(25n+r) &\equiv 0 \pmod{125}, \quad r\in \{18, 23\}, \label{p-1-3-mod25} \\
p_{-1/3}(361n+356) &\equiv 0 \pmod{361}, \label{p-1-3-mod361} \\
p_{-2/3}(49n+r) &\equiv 0 \pmod{7}, \quad r\in \{22, 29, 43\}, \label{p-2-3-mod7} \\
p_{-{3}/{4}}(25n+r) &\equiv 0 \pmod{25}, \quad r \in \{13, 23\}, \label{p-3-4-mod25} \\
p_{-3/4}(25n+18) & \equiv 0 \pmod{125}, \label{p-3-4-mod125} \\
p_{-3/4}(125n+r) &\equiv 0 \pmod{3125}, \quad r \in \{93, 118\}. \label{p-3-4-mod3125}
\end{align}
\end{conj}
\section{Modular Approach to Some Congruences}
It is possible to prove some congruences in Conjectures \ref{conj-1} and \ref{conj-2} using the theory of modular forms. We illustrate the method by giving a proof to \eqref{p-1-2-mod289}. Let
\begin{equation*}
\mathrm{SL}_{2}(\mathbf{Z}):=\Bigg\{\begin{pmatrix} a & b \\ c & d \end{pmatrix} \Big| a,b,c,d \in \mathbf{Z}, ad-bc=1\Bigg\}.
\end{equation*}
We denote by $M_{k}(\mathrm{SL}_{2}(\mathbf{Z}))$ (resp.\ $S_{k}(\mathrm{SL}_{2}(\mathbf{Z}))$) the space of modular forms (resp.\ cusp forms) of weight $k$ on $\mathrm{SL}_{2}(\mathbf{Z})$. For any positive integer $m$, we define the Hecke operator $T(m)$  and $U$-operator $U(m)$  which send a function
$$f(z)=\sum\limits_{n=0}^{\infty}{a(n)q^{n}}$$
to
$$f(z)|_{T(m)}:=\sum\limits_{n=0}^{\infty}\left(\sum\limits_{d|(m,n)}d^{k-1}a\left(\frac{nm}{d^2}\right)\right)q^n$$
and
$$f(z)|_{U(m)}:=\sum\limits_{n=0}^{\infty}a(mn)q^n,$$
respectively.
It is known that if $f(z)\in M_{k}(\mathrm{SL}_{2}(\mathbf{Z}))$, then $f(z)|_{T(m)}\in M_{k}(\mathrm{SL}_{2}(\mathbf{Z}))$.
\begin{proof}[Modular Proof of \eqref{p-1-2-mod289}]
Let $q=e^{2\pi i\tau}$ with $\mathrm{Im} \tau>0$. Recall the discriminant modular form
\begin{align}
\Delta(\tau):=q(q;q)_{\infty}^{24}.
\end{align}
It is clear that $\Delta^{6}(\tau) \in S_{72}(\mathrm{SL}_{2}(\mathbf{Z}))$ is a cusp form.
By Lemma \ref{binom-lem}, we deduce that
\begin{align}\label{add-1}
\Delta^6(\tau)&=q^6\frac{(q;q)_{\infty}^{289/2}}{(q;q)_{\infty}^{1/2}} \nonumber\\
&=(q;q)_{\infty}^{289/2}\sum_{n=0}^{\infty}p_{-1/2}(n)q^{n+6} \nonumber\\
&\equiv (q^{17};q^{17})_{\infty}^{17/2}\sum_{n=0}^{\infty}p_{-1/2}(n)q^{n+6}  \pmod{17^2}.
\end{align}
Applying the Hecke operator $T_{17}$ to both sides,  and observing that applying $T_{17}$ is the same as applying $U_{17}$ modulo $17^2$, we obtain
\begin{align}\label{add-2}
\Delta^6(\tau)|_{T_{17}} \equiv (q;q)_{\infty}^{17/2}\sum_{n=1}^{\infty}p_{-1/2}(17n-6)q^n \pmod{17^2}.
\end{align}
From \eqref{p-1-2-mod17}, we know that $p_{-1/2}(17n-6)\equiv 0$ (mod 17). By Lemma \ref{binom-lem}, we deduce from \eqref{add-2} that
\begin{align*}
\frac{1}{17}\Delta^6(\tau)|_{T_{17}} &\equiv (q;q)_{\infty}^{17/2}\sum_{n=0}^{\infty} \frac{p_{-1/2}(17n-6)}{17}q^n \\
&\equiv (q^{17};q^{17})_{\infty}^{1/2}\sum_{n=0}^{\infty} \frac{p_{-1/2}(17n-6)}{17}q^n
\pmod{17},
\end{align*}
or \begin{equation}\label{1-2-mod17}
\Delta^6(\tau)|_{T_{17}}
\equiv (q^{17};q^{17})_{\infty}^{1/2}\sum_{n=0}^{\infty} p_{-1/2}(17n-6)q^n
\pmod{17^2}.
\end{equation}
Since $\Delta^6(\tau)|_{T_{17}} \in S_{72}(\mathrm{SL}_{2}(\mathbf{Z}))$, we apply the Hecke operator $T_{17}$ to both sides of \eqref{1-2-mod17} and deduce that
\begin{align}\label{add-3}
\left(\Delta^6(\tau)|_{T_{17}}\right)|_{T_{17}}\equiv (q;q)_{\infty}^{1/2}\sum_{n=0}^{\infty}p_{-1/2}(17^2n-6)q^n \pmod{17^2}.
\end{align}
Now we recall the following Eisenstein series on $\mathrm{SL}_{2}(\mathbf{Z})$:
\begin{align}
E_{6}:=1-504\sum_{n=1}^{\infty}\frac{n^5q^n}{1-q^n}.
\end{align}
Let
\begin{align*}
B_{1}&:=\Delta^6(\tau), \quad B_2:=\Delta^5(\tau)E_{6}^2, \quad B_3:=\Delta^4(\tau)E_{6}^4, \\
B_{4}&:=\Delta^3(\tau)E_{6}^6, \quad B_5:=\Delta^2(\tau)E_{6}^8, \quad  B_6:=\Delta(\tau)E_{6}^{10}.
\end{align*}
It is not difficult to see that $\{B_1,B_2,B_3,B_4,B_5,B_6\}$ forms a basis of $S_{72}(\mathrm{SL}_{2}(\mathbf{Z}))$. By comparing the Fourier coefficients we find that
\begin{align}\label{last-Delta}
\left(\Delta^6(\tau)|_{T_{17}}\right)|_{T_{17}}=\sum_{i=0}^6a_{i}B_{i},
\end{align}
where
\begin{align*}
a_1=&2803266424444011486961793663394426123943306806893849573592292186093616 \\
& 946565526483482308, \\
a_2=&1113231602545024595543146596204782142754892610829246238990919796002850\\
&856740428953088, \\
a_3=&4732834266810238479570385785097996159241875744074623451960104362168616\\
&39045631744, \\
a_4=&-155407415188884349022329179204737911390822930792498205036684797032177\\
&77938560, \\
a_5=&-160448915469735241442136908278917088844111179846012597013010883109766\\
&72,\\
a_6=&216026225099443878192110703691596145681836890232383902466304.
\end{align*}
It is easy to verify that
\begin{align*}
\mathrm{ord}_{17}(a_1)=3, \quad \mathrm{ord}_{17}(a_i)=2, \quad 2\leq i \leq 6.
\end{align*}
From \eqref{add-3} and \eqref{last-Delta} we complete the proof of \eqref{p-1-2-mod289}.
\end{proof}

While we believe that this method is applicable to most of the congruences in Conjectures \ref{conj-1} and \ref{conj-2},  we are not sure if one can establish these congruences without the use of modular forms.

\section{Concluding Remarks}

Ramanujan's original proofs of \eqref{Rama-mod-5} and \eqref{Rama-mod-7} (see \cite{Rama1919}) involve the fourth and sixth powers of $(q;q)_\infty$.
In 1969,  Winquist \cite{Winquist} discovered an identity for $(q;q)_\infty^{10}$ and gave a proof of \eqref{Rama-mod-11} which is in the spirit of Ramanujan's proofs for \eqref{Rama-mod-5} and \eqref{Rama-mod-7}. Recently,  Hirschhorn \cite{Hirschhorn-JNT} gave a simple proof of \eqref{Rama-mod-11} that relies only on \eqref{eta-1} and \eqref{eta-3}. One common feature of the identities used by Ramanujan and Winquist is that for $d=4,6$ and 10,
$(q;q)_\infty^d$ can be expressed in the form
\begin{equation}\label{Qmn} \sum_{m,n=-\infty}^\infty A(m,n)q^{Q(m,n)},\end{equation} where $A(m,n)$ is a polynomial in $m$ and $n$ and
 $Q(m,n)$ is a degree 2 polynomial in $m$ and $n$. In 1985, J.P. Serre \cite{Serre-eta} showed that if $d$ is even, then
$(q;q)_\infty^d$ can be expressed in the form given by \eqref{Qmn} if and only if $d=2,4,6,8,10,14$ and 26. The proof of a series representations for $(q;q)_\infty^{26}$  was given for the first time in \cite{Serre-eta} although the identity in a different form was first discovered by A.O.L. Atkin (see \cite{Dyson}). For alternative representations of $(q;q)_\infty^{26}$, see the works \cite{Chan-Adv,Chan-JLMS} by Chan, S. Cooper and P.C. Toh.
In this work, we return to Ramanujan's original idea and derive congruences satisfied by $p_k(n)$ for certain rational number $k$ from  the series representations for $(q;q)_\infty^d$. In particular, it seems that this is the first time that expansions of $(q;q)_{\infty}^{14}$ and $(q;q)_{\infty}^{26}$ are associated to congruences analogous to Ramanujan's partition congruences \eqref{Rama-mod-5}--\eqref{Rama-mod-11}.

\subsection*{Acknowledgements}
We would like to thank Ernest X.W.\ Xia for his comments on an earlier version of this paper. We are also grateful to the referee for his/her suggestions. The second author is partially supported by ``the Fundamental Research Funds for the Central Universities'' and a start-up research grant of the Wuhan University.

\end{document}